\documentclass[12pt]{article}
\usepackage{amsmath}
\usepackage{amssymb}
\usepackage{amsthm}
\usepackage{mathabx}
\usepackage{caption}
\usepackage[usenames]{color}
\usepackage{amscd}
\usepackage{dsfont}
\usepackage{indentfirst}

\usepackage[colorlinks=true,linkcolor=blue,filecolor=red,
citecolor=webgreen]{hyperref}
\definecolor{webgreen}{rgb}{0,.5,0}

\numberwithin{equation}{section}

\def\C{{\mathds{C}}}

\def\N{{\mathds{N}}}
\def\Z{{\mathds{Z}}}
\def\1{{\bf 1}}

\def\a{{\bf a}}

\def\pont{$\bullet$ }
\newcommand{\DOT}{\text{\rm\Huge{.}}}

\newtheorem{theorem}{Theorem}[section]

\newtheorem{lemma}[theorem]{Lemma}

\begin{document}
	
	\title{{\bf Generalizations of Euler's $\varphi$-function with respect to systems
			of polynomials of several variables}}
	\author{Norbert Csizmazia and L\'aszl\'o T\'oth\thanks{corresponding author, e-mail: {\tt ltoth@gamma.ttk.pte.hu}}  \\
		Department of Mathematics \\
		University of P\'ecs \\
		Ifj\'us\'ag \'utja 6, 7624 P\'ecs \\ 
		Hungary}
	\date{}
	\maketitle

	\centerline{Acta Universitatis Sapientiae, Mathematica {\bf 18} (2026), Article 31, 18 pp.}
	
	\begin{abstract} We introduce a new generalization of Euler's $\varphi$-function associated with a system of 
		polynomials of several variables. We reprove by a short direct approach certain known related identities, and
		study some other special cases that do not appear in the literature. We also discuss the unitary analogues of 
		these functions. 
	\end{abstract}
	
	

	\section{Introduction}
	Euler's arithmetic function $\varphi(n)$ is defined as the number of integers $x\in \{1,\ldots,n\}$ such that $\gcd(x,n)=1$. Jordan's function $J_k(n)$ of order $k$, where $k\in \N:=\{1,2,\ldots\}$, is a generalization of Euler's function, and is defined as the number of ordered $k$-tuples
	$(x_1,...,x_k) \in \{1,\ldots,n\}^k$ such that $\gcd(x_1,\ldots,x_k,n)=1$. 
	Here $J_1(n)=\varphi(n)$.
	
	Euler's function $\varphi(n)$ (throughout the paper we will use this notation) has several other 
	generalizations in the literature. See, e.g., the monographs by S\'andor and Crstici \cite[Sect.\ 3.7]{SC2004} and  Sivaramakrishnan \cite[Ch.\ V]{Siv1989}. Let $f(x)$ be a polynomial with integer coefficients. Menon 
	\cite{Men1967} defined the function $\phi_f(n)$ as the number of integers $x\in \{1,\ldots,n\}$
	such that $\gcd(f(x),n)=1$. If $f(x)=x$, then this recovers Euler's function $\varphi(n)$. The functions
	introduced independently of Menon's paper by the following authors are all special cases of Menon's function, corresponding to the given polynomials: 
	
	\pont Alder \cite{Ald1958}, Nagell \cite{Nag1923}: $f(x)=x(m-x)$, $m\in \N\cup \{0\}$, 
	
	\pont Garcia, Ligh \cite{GL1983}: $f(x)=s+(x-1)d$, $s$ and $d$ being relatively prime integers, 
	
	\pont Schemmel \cite{Sch1869}: $f(x)= x(x-1)\cdots (x-m+1)$, $m\in \N$.
	
	However, Menon's function does not include Jordan's function of order $k$. Given a system
	$S=(f_1(x),\ldots,f_k(x))$ of $k$ polynomials with integer coefficients,
	Stevens \cite{Ste1971} defined, without citing the paper by Menon \cite{Men1967}, the function $\phi_S(n)$ 
	as the number of ordered $k$-tuples $(x_1,\ldots,x_k)\in \{1,\ldots,n\}^k$ such that $\gcd(f_1(x_1),\ldots,f_k(x_k),n)=1$. If $f_1(x)=\cdots = f_k(x)=x$, then this recovers the function $J_k(n)$. Several other special cases can be discussed. Also, see Chidambaraswamy \cite{Chi1974,Chi1979}.
	
	Calder\'on et al. \cite{Cald_etal2015} introduced and studied the $k$-dimensional ($k\in \N$) generalized Euler function $\phi_k(n)$, defined as the number of ordered $k$-tuples
	$(x_1,\ldots,x_k)\in \{1,\ldots,n\}^k$ such that $\gcd(x_1^2+\cdots+x_k^2,n)=1$. Note that $\phi_2(n)$ is the 
	restriction to the set of positive integers of the Euler function defined for the set $\Z[i]$ of Gaussian integers. Also, $\phi_4(n)$ and $\phi_8(n)$ count, respectively, the number of invertible quaternions and octonions (mod $n$). If $k=1$, then $\phi_1(n)=\varphi(n)$.
	
	The second author \cite{Tot2022} introduced and investigated the generalized Euler function 
	$\varphi_k(n)$ ($k\in \N$), given as the number of ordered $k$-tuples
	$(x_1,\ldots,x_k)\in \{1,\ldots,n\}^k$ such that both the product $x_1\cdots x_k$ and the sum $x_1+\cdots +x_k$ are prime to $n$. For $k=1$ we have $\varphi_1(n)=\varphi(n)$.
	
	Observe that the functions $\phi_k(n)$ and $\varphi_k(n)$ are not recovered by Stevens' generalization.
	
	In the present paper, we introduce the following new generalization of Euler's $\varphi$-function. Let $F=(f_1(x_1,\ldots, x_k), \ldots, f_m(x_1,\ldots,x_k))$ be a system of arbitrary (nonconstant) polynomials of $k$ variables with integer coefficients, 
	where $k,m\in \N$. We define the generalized Euler function $\Phi_F(n)$ as 
	\begin{equation} \label{gen_Euler_def}
		\Phi_F(n)= \# \{ (x_1,\ldots,x_k)\in \{1,\ldots,n\}^k: \gcd(f_1(x_1,\ldots, x_k), \ldots, f_m(x_1,\ldots,x_k),n)=1\}.
	\end{equation}
	
	The functions $\phi_k(n)$ by Calder\'on et al. \cite{Cald_etal2015} and $\varphi_k(n)$ by the second author \cite{Tot2022} are obtained from \eqref{gen_Euler_def} in the case $m=1$ and $f_1(x_1,\ldots,x_k)= x_1^2+\cdots +x_k^2$, respectively  $f_1(x_1,\ldots,x_k)= x_1\cdots x_k (x_1+\cdots + x_k)$.
	
	We deduce general results for the function $\Phi_F(n)$, reprove by a short direct approach the identities 
	computing $\phi_k(n)$ and $\varphi_k(n)$ given in \cite{Cald_etal2015,Tot2022}, and consider some other special cases 
	that are not given in the literature. For example, we investigate the function corresponding to a system of linear polynomials 
	$f_i(x_1,\ldots,x_k)= a_{i1}x_1+\cdots +a_{ik}x_k+b_i$ ($1\le i\le k=m$), and the case of a quadratic polynomial $f(x_1,\ldots,x_k)= 
	a_1x_1+\cdots +a_kx_k+b$ ($m=1$). 
		
	We also discuss the unitary analogues of some of these functions. We note that it is possible to further generalize these 
	functions in the context of regular convolutions of Narkiewicz-type, but we will not go into details. We refer to the second author and Haukkanen \cite{TotHau1996} for a generalization of Stevens' function with respect to regular convolutions.
	
	The paper is organized as follows. The main general results are included in Section \ref{Section_General_results}.
	Special cases of the function $\Phi_F(n)$ are discussed in Section \ref{Section_Special_cases}. Unitary analogues are presented in Section \ref{Section_Unitary_analogues}. Some lemmas are given in Section \ref{Section_Lemmas}, while the proofs of the 
	theorems are included in Section \ref{Section_Proofs}. Finally, some more remarks are presented in Section \ref{Sect_Further_remarks}. 
	
	\section{General results} \label{Section_General_results}
	
	Let $N_F(n)$ denote the number of solutions $(x_1,\ldots, x_k)\in \{1,\ldots,n\}^k$
	of the system of congruences $f_1(x_1,\ldots, x_k)\equiv 0$ (mod $n$), $\ldots$, $f_m(x_1,\ldots, x_k)\equiv 0$ (mod $n$). It can be shown by the Chinese remainder theorem
	that the function $n\mapsto N_F(n)$ is multiplicative. Let $\mu$ denote the M\"obius function.
	
	\begin{theorem} \label{Th_main} For every system $F$ of polynomials the function $n\mapsto \Phi_F(n)$ is multiplicative and for every $n\in \N$, 
		\begin{align}
			\Phi_F(n) & = n^k \sum_{d\mid n} \frac{\mu(d) N_F(d)}{d^k} \label{eq_1}  \\
			& = n^k \prod_{p\mid n} \left(1-\frac{N_F(p)}{p^k} \right). \label{eq_2}
		\end{align}
	\end{theorem}
	
	Identity \eqref{eq_2} shows that for a given system $F$ of polynomials it is enough to know the values $N_F(p)$ for 
	the primes $p$ to obtain explicit formulas for the
	corresponding generalized Euler functions. Some special cases are discussed in Section \ref{Section_Special_cases}.
	
	We note that it is possible to investigate the more general function
	\begin{equation*}
		P_{F,h}(n)= \sum_{(x_1,\ldots,x_k) \in \{1,\ldots,n\}^k} h(\gcd(f_1(x_1,\ldots, x_k), \ldots, f_m(x_1,\ldots,x_k),n)),
	\end{equation*}
	where $h:\N \to \C$ is an arbitrary function. The function $\Phi_F(n)$ is recovered by choosing $h(1)=1$, $h(n)=0$ ($n>1$). In the special case $k=m=1$, $f_1(x_1)=x_1$ and $h(n)=n$ ($n\in \N$) we obtain the gcd-sum function (Pillai function) 
	\begin{equation*}
		P(n)= \sum_{x=1}^n \gcd(x,n),
	\end{equation*}
	see the survey by the second author \cite{Tot2010}.
	
	\begin{theorem} \label{Th_main_h} Let $F$ be an arbitrary system of polynomials and $h$ be an arithmetic function. For every $n\in \N$,
		\begin{align} \label{form_h}
			P_{F,h}(n) = n^k \sum_{d\mid n} \frac{(\mu*h)(d) N_F(d)}{d^k}, 
		\end{align}
		where $*$ denotes the convolution. If $h$ is multiplicative, then the function $n\mapsto P_{F,h}(n)$ is also multiplicative. 
	\end{theorem}
	
	In this paper we do not investigate further properties and special cases of the function $P_{F,h}(n)$.
	An asymptotic formula with remainder term for the Euler-type function $\Phi_F(n)$ is given in the following theorem.
	
	\begin{theorem} \label{Th_asympt} Assume that the system of polynomials $F$ is such that for every $1\le i\le m$ the coefficients of $f_i(x_1,\ldots,x_k)$ are relatively prime. Then
		\begin{align} \label{asympt_form}
			\sum_{n\le x} \Phi_F(n) = x^{k+1} \prod_p \left(1-\frac{N_F(p)}{p^{k+1}} \right) +
			O(x^{k+1-1/t+\varepsilon}),
		\end{align}
		where $t$ is the minimum of degrees of the polynomials $f_i(x_1,\ldots,x_k)$ \textup{($1\le i\le m$)} and $\varepsilon > 0$ is arbitrary.
	\end{theorem}
	
	For special systems of polynomials $F$ the above remainder term can be improved.
	For example, the authors of paper \cite{Cald_etal2015} showed that for the function $\phi_k(n)$, corresponding to the case  $m=1$, $f_1(x_1,\ldots,x_k)=x_1^2+\cdots + x_k^2$, the error term is $O(x^k\log x)$ for $k$ even, and is $O(x^k(\log x)^{2/3}(\log \log x)^{4/3})$ for $k$ odd. This latter error term is a consequence of the error for Euler's function $\varphi(n)$ due to Walfisz. Also, according to paper \cite{Tot2022}, the error for the function $\varphi_k(n)$, obtained in the case $m=1$, $f_1(x_1,\ldots,x_k)=x_1\cdots x_k(x_1+\cdots + x_k)$, the error is $O(x^k (\log x)^{k+1})$ for every $k\ge 2$.
	
	For Stevens' function an asymptotic formula with remainder term $O(x^k)$ ($k\ge 2$) and $O(x^{1+\varepsilon})$ ($k=1$)
	has been obtained by the second author and S\'andor \cite{TotSan1989}. Also, see \cite{TotHau1996}. In the case of the function by Garcia and Ligh \cite{GL1983} the error is $O(x\log x$), see the second author \cite{Tot1987}. For the Schemmel function with parameter $m$ the error term is $O(x(\log x)^m)$,
	see Steuding and Suriajaya \cite{SS2023}.
	
	We also remark that some different asymptotic properties of the function $\phi_k(n)$ have recently been studied 
	by Banerjee et al. \cite{Ban_etal2024}.
	
	\section{Special cases} \label{Section_Special_cases}
	
	In what follows, we investigate special cases of the function $\Phi_F(n)$. 
	Using \eqref{eq_2} we can deduce explicit formulas to compute their values.
	First consider the case $k=m$ and the system of linear polynomials $f_i(x_1,\ldots,x_k)=a_{i1}x_1+\cdots +a_{ik}x_k+b_i$,
	where $a_{ij},b_i\in \Z$ ($1 \le i,j\le k$). Let $A=(a_{ij})_{1\le i,j\le k}$, $B= (b_1,\ldots,b_k)$, and let $\phi_{A,B}(n)$ denote the corresponding Euler-type function defined by \eqref{gen_Euler_def}.
	
	\begin{theorem} \label{Th_varphi_lin_gen} Let $k\in \N$ and assume that the matrix $A$ is unimodular, that is, $\det(A)=\pm 1$. Then for every $B \in \Z^k$ and every $n\in \N$ we have
		\begin{equation*}
			\phi_{A,B}(n) =n^k\prod_{p\mid n} \left(1-\frac1{p^k} \right) = J_k(n),
		\end{equation*}
		the Jordan function of order $k$.
	\end{theorem}
	
	Next, consider the function $\varphi_k(n)$ defined in the Introduction.
	This is the special case of the function $\Phi_F(n)$ corresponding to $m=1$ 
	and the polynomial $f_1(x_1,\ldots,x_k)= x_1\cdots x_k (x_1+\cdots + x_k)$.
	
	\begin{theorem} \label{Th_varphi_k} For every $k,n\in \N$,
		\begin{align} 
			\varphi_k(n) & = n^k \prod_{p\mid n} \left(1-\frac1{p} \right) \left(\left(1-\frac{1}{p}\right)^k - \frac{(-1)^{k}}{p^k} \right)  \label{varphi_k_form} \\
			&  = n^{k-1} \varphi(n) \prod_{p\mid n} \left(\left(1-\frac{1}{p}\right)^k - \frac{(-1)^{k}}{p^k} \right).
			\nonumber 
		\end{align}
	\end{theorem}
	
	Identity \eqref{varphi_k_form} has been obtained by the second author \cite[Th.\ 2.1]{Tot2022} using
	a different method, namely induction on $k$.
	
	Now, let $m=2$ and $f_1(x_1,\ldots,x_k)= x_1\cdots x_k$, $f_2(x_1,\ldots,x_2)=x_1+\cdots + x_k$, which leads to another generalization of the Euler function, let us denote it by $\varphi_k'(n)$, 
	not given in the literature. That is,
	\begin{align*}
		\varphi_k'(n)= \# \{ (x_1,\ldots,x_k) \in \{1,\ldots,n\}^k:  \gcd(x_1\cdots x_k, x_1+\cdots +x_k,n)=1\}. 
	\end{align*}
	
	We have the following result.
	
	\begin{theorem} \label{Th_varphi_k_apost} For every $k,n\in \N$ we have
		\begin{align} 
			\varphi_k'(n) & = n^k \prod_{p\mid n} \left(1-\frac1{p}\right) \left(1+ \frac{1}{p}\left(1-\frac1{p}\right)^{k-1} + \frac{(-1)^{k}}{p^k} \right)  \label{Varphi_k_apost} \\
			& = n^{k-1} \varphi(n) \prod_{p\mid n} \left(1+ \frac{1}{p}\left(1-\frac1{p}\right)^{k-1} + \frac{(-1)^{k}}{p^k} \right). \nonumber
		\end{align}
	\end{theorem}
	
	Note that $\varphi_1'(n)=\varphi(n)$, $\varphi_2'(n)=J_2(n)$ and $\varphi_3'(n)=n^3 \prod_{p\mid n} \left(1- 1/p\right)^2\left(1+2/p\right)$.
	
	Let $\Phi^{(s)}_{k,\a,b}(n)$ be defined as the number of ordered $k$-tuples $(x_1,\ldots,x_k)\in \{1,\ldots,n\}^k$
	such that $\gcd(a_1x_1^s+\cdots + a_kx_k^s+b,n)=1$, where $k,s\in \N$ and $\a=(a_1,\ldots,a_k)\in \Z^k$, $b\in \Z$.
	This corresponds to the case $m=1$ and $f(x_1,\ldots,x_k)=a_1x_1^s+\cdots + a_kx_k^s+b$.
	
	In the case $s=1$ we have the next result.
	
	\begin{theorem} \label{Th_Phi_k_linear}
		Let $k\in \N$, $\a=(a_1,\ldots,a_k)\in \Z^k$, $b\in \Z$, and $n\in \N$. Then 
		\begin{align} \label{id_linear}
			\Phi^{(1)}_{k,\a,b}(n) = n^k \prod_{\substack{p\mid n\\ p\, \nmid \, \gcd(a_1,\ldots,a_k)}} \left(1-\frac1{p} \right),
		\end{align}
		provided that for every $p\mid n$ one has $p\nmid \gcd(a_1,\ldots,a_k,b)$. In the contrary case $\Phi^{(1)}_{k,\a,b}(n)=0$.
	\end{theorem}
	
	It follows that in the case $\gcd(a_1,\ldots,a_k)=1$, $b\in \Z$ we have $\Phi^{(1)}_{k,\a,b}(n) = n^{k-1}\varphi(n)$ for every $n\in \N$.
	
	Now consider the function $\Phi^{(2)}_{k,\a,b}(n)$ corresponding to the case $s=2$. It recovers the function $\Phi_k(n)$ by Calder\'on et al. \cite{Cald_etal2015} if $a_1=\cdots =a_k=1$ and $b=0$. Here we assume that $b=0$, and prove the next result.
	
	\begin{theorem} \label{Th_Phi_k_squares} Let $k\in \N$, $\a=(a_1,\ldots,a_k)\in \Z^k$ and $n\in \N$ such that $\gcd(a_1\cdots a_k,n)=1$. 
		
		If $k$ is odd, then
		\begin{align} \label{id_square_odd}
			\Phi^{(2)}_{k,\a,0}(n)  = n^{k-1} \varphi(n), 
		\end{align}
		and if $k$ is even, then 
		\begin{align} \label{id_square_even}
			\Phi^{(2)}_{k,\a,0}(n)  = n^{k-1} \varphi(n) \prod_{\substack{p\mid n\\ p>2}} \left(1-\frac1{p^{k/2}}(-1)^{k(p-1)/4} \left( \frac{a_1}{p} \right) \cdots \left(\frac{a_k}{p}\right) \right), 
		\end{align}
		where $\left( \frac{a_i}{p} \right)$ denote the Legendre symbols \textup{($1\le i \le k$)}.
	\end{theorem}
	
	In the case $a_1=\cdots =a_k=1$ this reduces to \cite[Th.\ 1]{Cald_etal2015}, obtained by some different arguments.
	
	Now consider the function
	\begin{align*}
		\phi_k'(n)= \# \{ (x_1,\ldots,x_k) \in \{1,\ldots,n\}^k:  \gcd(x_1\cdots x_k, x_1^2+\cdots +x_k^2,n)=1\}, 
	\end{align*}
	not given in the literature, corresponding to the case $m=2$ and $f_1(x_1,\ldots,x_k)= x_1\cdots x_k$, $f_2(x_1,\ldots,x_2)=x_1^2+\cdots + x_k^2$. Here $\phi_1'(n)= \varphi(n)$.
	
	\begin{theorem} \label{Th_Phi_k_last}
		For every $k\in \N$ and $n\in \N$, 
		\begin{align*} 
			\phi_k'(n) = n^k \prod_{p\mid n} \left(1- \frac{T_k(p)}{p^k}\right),
		\end{align*}
		where
		\begin{align*}
			T_k(p)= \frac1{p} \left(p^k - (p-1)^k + (-1)^{k-1}(p-1) \sum_{s=0}^{\lfloor(k-1)/2\rfloor} (-1)^{s(p-1)/2} \binom{k}{2s} p^s \right)
		\end{align*}
		for $p>2$, and $T_k(2)=2^{k-1}$ for $k$ odd, $T_k(2)=2^{k-1}-1$ for $k$ even.
	\end{theorem}
	
	We remark that $\phi'_2(n) = J_2(n)$ ($n\in \N$), the Jordan function of order two. 
	
	\section{Unitary analogues} \label{Section_Unitary_analogues}
	
	We recall that $d$ is a unitary divisor of $n$ if $d\mid n$ and $\gcd(d,n/d)=1$, notation $d\mid \mid n$. Let $(a,b)_*$ denote, as usual, the largest divisor of $a$ which is a unitary divisor of $b$. The unitary analogue $\varphi^*(n)$ of Euler's function is defined as the number of integers $x\in \{1,\ldots,n\}$ such that $(x,n)_*=1$. The function $\varphi^*$ is multiplicative and one has $\varphi^*(n)=n\prod_{p^{\nu} \mid\mid n} (1-1/p^{\nu})$. Unitary analogues of the Jordan and Stevens functions are known in the literature; see Nageswara Rao \cite{Nag1966}, Chidambaraswamy \cite{Chi1979}.
	
	For a given system $F=(f_1(x_1,\ldots, x_k), \ldots, f_m(x_1,\ldots,x_k))$ of arbitrary (nonconstant) polynomials of $k$ variables with integer coefficients ($k,m\in \N$) we define the unitary analogue $\Phi_F^*(n)$ of the function $\Phi_F(n)$ as
	follows: 
	\begin{equation} \label{gen_unit_Euler_def}
		\Phi_F^*(n)= \# \{ (x_1,\ldots,x_k)\in \{1,\ldots,n\}^k: (\gcd(f_1(x_1,\ldots, x_k), \ldots, f_m(x_1,\ldots,x_k)),n)_*=1\}.
	\end{equation}
	
	In the special cases of $F$ discussed above this leads to the unitary analogues of the functions we have investigated. 
	
	The unitary M\"obius function $\mu^*$ is defined as the inverse of the constant $1$ function with respect to the 
	unitary convolution. Here $\mu^*(n)=(-1)^{\omega(n)}$, 
	where $\omega(n)$ is the number of distinct prime factors of $n$. See, e.g., Cohen \cite{Coh1960}, Sivaramakrishnan 
	\cite[Sect.\ 2.1]{Siv1989}.
	
	\begin{theorem} \label{Th_main_unitary} For every system $F$ of polynomials the function $n\mapsto \Phi_F^*(n)$ is multiplicative and for every $n\in \N$, 
		\begin{align}
			\Phi_F^*(n) & = n^k \sum_{d\mid\mid n} \frac{\mu^*(d) N_F(d)}{d^k} \label{eq_1_unit}  \\
			& = n^k \prod_{p^\nu \mid \mid n} \left(1-\frac{N_F(p^\nu)}{p^{\nu k}} \right). \label{eq_2_unit}
		\end{align}
	\end{theorem}
	
	Identity \eqref{eq_2_unit} shows that it is necessary to know the values $N_F(p^\nu)$ for prime powers $p^\nu$ ($\nu \ge 1$) to obtain explicit formulas for the corresponding generalized Euler functions.
	
	We present only the following special cases. We leave it to the interested readers to develop other such unitary Euler-type functions. 
	
	Consider the unitary Euler-type function $\phi^{*}_{A,B}(n)$, defined by \eqref{gen_unit_Euler_def},
	corresponding to $k=m$ and the system of linear polynomials $f_i(x)=a_{i1}x_1+\cdots +a_{ik}x_k+b_i$,
	where $A=(a_{ij})_{1\le i,j\le k}$ is an integer matrix and $B= (b_1,\ldots,b_k) \in \Z^k$.
	
	\begin{theorem} \label{Th_varphi_unit_lin_gen} Let $k\in \N$ and assume that the matrix $A$ is unimodular, that is,  $\det(A)=\pm 1$. Then for every $B= (b_1,\ldots,b_k) \in \Z^k$ and every $n\in \N$ we have
		\begin{equation*}
			\phi^{*}_{A,B}(n) =n^k\prod_{p^\nu \mid \mid n} \left(1-\frac1{p^{\nu k}} \right) = J^{*}_k(n),
		\end{equation*}
		the unitary Jordan function of order $k$.
	\end{theorem}
	
	Let $\varphi_k^*(n)$ be defined as
	\begin{equation*} 
		\varphi_k^*(n)= \# \{ (x_1,\ldots,x_k)\in \{1,\ldots,n\}^k: (\gcd(x_1\cdots x_k, x_1+\cdots +x_k),n)_*=1\},
	\end{equation*}
	being the unitary analogue of the function $\varphi_k(n)$ given in the Introduction. Here $\varphi_1^*(n)= \varphi^*(n)$. In the case $k=2$ we have the next result.
	
	\begin{theorem} \label{Th_unitary_2} For every $n\in \N$, 
		\begin{align*}
			\varphi_2^*(n)  = n^2 \prod_{p^{2\nu} \mid \mid n} \left(1-\frac1{p^{3\nu}} \right)
			\prod_{p^{2\nu-1} \mid \mid n} \left(1-\frac1{p^{3\nu-1}} \right).
		\end{align*}
	\end{theorem}
	
	Let $\Phi^{(s)*}_{k,\a,b}(n)$ be defined as the number of ordered $k$-tuples $(x_1,\ldots,x_k)\in \{1,\ldots,n\}^k$
	such that $(a_1x_1^s+\cdots + a_kx_k^s+b,n)_*=1$, where $k,s\in \N$ and $\a=(a_1,\ldots,a_k)\in \Z^k$, $b\in \Z$.
	This corresponds to the case $m=1$ and $f(x_1,\ldots,x_k)=a_1x_1^s+\cdots + a_kx_k^s+b$.
	
	In the case $s=1$ we have the next result.
	
	\begin{theorem} \label{Th_Phi_k_linear_unit}
		Let $k\in \N$, $\a=(a_1,\ldots,a_k)\in \Z^k$, $b\in \Z$, and $n\in \N$. Then 
		\begin{align*} 
			\Phi^{(1)*}_{k,\a,b}(n) & = n^k \prod_{\substack{p^\nu \mid \mid n\\ \gcd(a_1,\ldots,a_k,p^{\nu})\mid b}} \left(1-\frac{\gcd(a_1,\ldots,a_k,p^\nu)}{p^\nu} \right).
		\end{align*}
	\end{theorem}
	
	It follows that $\Phi^{(1)*}_{k,\a,b}(n)=0$ if and only if there exists $p^\nu \mid \mid n$ such that 
	$p^\nu \mid \gcd(a_1,\ldots,a_k,b)$. If $\gcd(a_1,\ldots,a_k)=1$ and $b=0$, then $\Phi^{(1)*}_{k,\a,0}(n) = 
	n^{k-1}\varphi^*(n)$ for every $n\in \N$.
	
	The case $s=2$ can be treated using known explicit formulas for the number $N_k(\a,b,n)$ of solutions of the quadratic congruence $a_1x_1^2+\cdots + a_kx_k^2\equiv b$ (mod $n$). See Section \ref{Sect_Further_remarks} for more details.
	
	\section{Lemmas} \label{Section_Lemmas}
	
	To estimate the error term of the asymptotic formula of Theorem \ref{Th_asympt}, we need the following lemmas.
	
	\begin{lemma} {\rm (Hua \cite[Lemma\ 2.3]{Hua1965})} \label{Lemma_Hua}
		Let $f(x_1, \ldots,x_k)$ be a polynomial with integer coefficients of $k$ variables of degree $t$ \textup{($k,t\ge 1$)} such that not all of its coefficients are divisible with the prime $p$. Let $N_f(p^\nu)$ denote the number of solutions of the congruence 
		\begin{align*}
			f(x_1, \ldots, x_k) \equiv 0 \pmod{p^\nu}
		\end{align*}
		of the prime power modulus $p^\nu$ \textup{($\nu \ge 1$)}. Then 
		\begin{align*}
			N_f(p^{\nu}) \leq C (\nu + 1)^{k-1} p^{\nu (k-1/t)},
		\end{align*}
		where $C=C(k,t)$ is a constant depending only on $k$ and $t$.
	\end{lemma}
	
	\begin{lemma} \label{Lemma_Hua_kov} Assume that the system of polynomials $F$ is such that for every $1\le i\le m$ the coefficients of $f_i(x_1,\ldots,x_k)$ are relatively prime. Let $t$ be the minimum of degrees of the polynomials $f_i(x_1,\ldots,x_k)$ \textup{($1\le i\le m$)}. Then for every $\varepsilon>0$, 
		\begin{equation*}
			N_F(n) = O(n^{k-1/t+\varepsilon}).
		\end{equation*} 
	\end{lemma}
	
	\begin{proof}[Proof of Lemma {\rm \ref{Lemma_Hua_kov}}] 
		The number $N_F(p^{\nu})$ of solutions of the system of congruences 
		\begin{align*}
			f_i(x_1,\ldots,x_k)\equiv 0 \quad \text{(mod $p^{\nu}$)} \quad (1\le i\le m)
		\end{align*}
		is at most the number of solutions of any of the congruences $f_i(x_1,\ldots,x_k)\equiv 0$ (mod $p^{\nu}$). Select the congruence corresponding to the polynomial $f_i(x_1,\ldots,x_k)$ 
		of the smallest degree. This will give the best estimate using Lemma \ref{Lemma_Hua}. We obtain 
		\begin{align*}
			N_F(p^\nu) \leq C (\nu + 1)^{k-1} p^{\nu (k-1/t)},
		\end{align*} 
		and by the multiplicativity of $N_F(n)$, 
		\[
		N_F(n) \leq C^{\omega(n)} \tau(n)^{k-1} n^{k-1/t}, 
		\]
		where $\omega(n)$ is the number of distinct prime factors of $n$, and 
		$\tau(n)= \sum_{d\mid n} 1$. Using the familiar estimate $\tau(n)=O(n^{\varepsilon})$ and its consequence
		$C^{\omega(n)} =(2^{\omega(n)}){^{\log_2 C}}\le \tau(n)^{\log_2 C} = O(n^{\varepsilon})$ we deduce that
		\begin{align*}
			N_F(n) =O(n^{k-1/t + \varepsilon}),
		\end{align*}
		where $k$ is fixed and $\varepsilon > 0$ is arbitrary. 
	\end{proof}
	
	\begin{lemma}  \label{Lemma_sum} Let $s$ be a real number. One has the well known estimates 
		\begin{align*}
			\sum_{n\le x} n^s & = \frac{x^{s+1}}{s+1}+ O(x^s), \quad s\ge 0, \\    
			\sum_{n\le x} n^{-s} & = O(x^{1-s}), \quad 0<s<1, \\ 
			\sum_{n>x} n^{-s} & = O(x^{1-s}), \quad s>1.
		\end{align*}
	\end{lemma}
	
	Next, we deduce the number of solutions $(x_1,\ldots,x_k)$ of some congruences and systems of congruences (mod $p$). 
	Throughout, we assume that $(x_1,\ldots,x_k)\in \{1,\ldots,p\}^k$.
	
	\begin{lemma} \label{Lemma_syst_lin_cong} 
		Consider the system of linear congruences 
		\begin{align*}
			a_{11}x_1+\cdots +a_{1k}x_k+b_1 \equiv & \ 0 \textup{ (mod $n$)} \\ \ldots & \ldots \\
			a_{k1}x_1+\cdots +a_{kk}x_k+b_k \equiv & \ 0 \textup{ (mod $n$)},
		\end{align*}
		where $a_{ij},b_i\in \Z$ \textup{($1 \le i,j\le k$)} and $n\in \N$. Let $A=(a_{ij})_{1\le i,j\le k}$ and assume that 
		$\gcd(\det(A),n)=1$. Then the system has a unique solution \textup{(mod $n$)}.
	\end{lemma}
	
	\begin{proof} See, e.g., the book by Rosen \cite[Sect.\ 4.5]{Ros2011}.
	\end{proof}

	\begin{lemma} \label{Lemma_cong_prod} Let $k\in \N$ and $p$ be a prime.
		The number $N_k(p)$ of solutions of the congruence
		$x_1\cdots x_k (x_1+\cdots + x_k)\equiv 0$ \textup{ (mod $p$)} 
		is
		\begin{align*}
			N_k(p)= p^k- \frac{p-1}{p}\left((p-1)^k+(-1)^{k-1}\right).
		\end{align*}
	\end{lemma}
	
	\begin{proof}[Proof of Lemma {\rm \ref{Lemma_cong_prod}}]
		To deduce the number $N_k(p)$ of solutions of the given congruence,
		we consider the following cases.
		
		Case I. One of the values $x_1,\ldots,x_k$ is $p$. The number $A_k(p)$ of these solutions is
		the difference between the total number of $k$-tuples $(x_1,\ldots,x_k)$ and the number of
		$k$-tuples where none of the values $x_1,\ldots,x_k$ is $p$, that is, $A_k(p)= p^k-(p-1)^k$.
		
		Case II. If $1\le x_1,\ldots,x_k\le p-1$, then to deduce the number $B_k(p)$ of solutions of the congruence
		$x_1+\cdots +x_k\equiv 0$ (mod $p$) we use the identity 
		\begin{align} \label{id_exp_prime}
			\sum_{1\le a\le p} e^{2\pi ija/p} =\begin{cases} p, \text{ if $p\mid j$},\\ 0, \text{ otherwise},
			\end{cases}   
		\end{align}
		where $j\in \N$, and obtain
		\begin{equation*}
			B_k(p)= \sum_{\substack{1\le x_1,\ldots,x_k\le p-1 \\ x_1+\cdots+x_k\equiv 0 \text{ (mod $p$)}}} 1 =
			\frac1{p} \sum_{1\le x_1,\ldots,x_k\le p-1} \sum_{j=1}^p e^{2\pi ij(x_1+\cdots+x_k)/p}
		\end{equation*}
		\begin{equation*}
			= \frac1{p} \sum_{j=1}^p \sum_{1\le x_1\le p-1} e^{2\pi ijx_1/p} \cdots \sum_{1\le x_k\le p-1} e^{2\pi ijx_k/p} 
		\end{equation*}
		\begin{equation*}
			= \frac1{p} \sum_{j=1}^p \left(\sum_{1\le x_1\le p} e^{2\pi ijx_1/p}-1\right) \cdots \left(\sum_{1\le x_k\le p} 
			e^{2\pi ijx_k/p}-1\right)  
		\end{equation*}
		\begin{equation*}
			= \frac1{p} \left( \sum_{j=1}^{p-1} (-1)^k + (p-1)^k\right) = \frac1{p}\left((p-1)(-1)^k+(p-1)^k \right),
		\end{equation*}
		using \eqref{id_exp_prime}, again.
		This gives 
		\begin{align*}
			N_k(p)= A_k(p)+B_k(p) & =p^k-(p-1)^k + \frac1{p}\left((p-1)(-1)^k+(p-1)^k \right) \\
			& =p^k- \frac{p-1}{p}\left((p-1)^k+(-1)^{k-1}\right),
		\end{align*} 
		finishing the proof.
	\end{proof}
	
	\begin{lemma} \label{Lemma_cong_system} Let $k\in \N$ and let $p$ be a prime. The number $N'_k(p)$ of solutions 
		of the system of congruences  
		$x_1\cdots x_k \equiv 0$ \textup{(mod $p$)}, $x_1+\cdots + x_k \equiv 0$ \textup{(mod $p$)} is
		\begin{align*}
			N'_k(p)= p^{k-1} - \frac1{p}\left((p-1)^k+(-1)^k(p-1)\right).    
		\end{align*}
	\end{lemma}
	
	\begin{proof}[Proof of Lemma {\rm \ref{Lemma_cong_system}}]
		Here $N'_k(p)$ is exactly the number of solutions of the congruence $x_1+\cdots + x_k \equiv 0$ (mod $p$), where at least one of $x_1,\ldots,x_k$ is $p$. This is the difference between the number of all solutions of this congruence and the number $C_k(p)$ of solutions, where none of $x_1,\ldots,x_k$ is $p$.
		The number of all solutions is $p^{k-1}$, since $x_1,\ldots,x_{k-1}$ can be selected arbitrarily and $x_k\equiv -(x_1+\cdots +x_k)$ (mod $p$). The value $C_k(p)$ has been found above in the proof of Lemma \ref{Lemma_cong_prod}, and is $C_k(p)=B_k(p)= \frac1{p}((p-1)^k+(-1)^k(p-1))$. This gives 
		\begin{align*}
			N'_k(p)= p^{k-1} - \frac1{p}\left((p-1)^k+(-1)^k(p-1)\right).    
		\end{align*}
		
		Alternatively, to find $N'_k(p)$, we can use the inclusion-exclusion principle.
		Let $A_i:=\{(x_1,\ldots,x_k)\in \{1,\ldots,p\}^k: p\mid x_i \text{ and } p\mid (x_1+\cdots +x_k)\}$ ($1\le i \le k$). Then $|A_i|=p^{k-2}$ ($1\le i\le k$), since we need $x_i=p$ and all but one of the variables $x_j$ with $j\ne i$ can be arbitrarily selected. Also, $|A_i\cap A_j|= p^{k-3}$ ($1\le i<j\le k$), $|A_i\cap A_j\cap A_{\ell}|=p^{k-4}$ ($1\le i<j<\ell\le k$), ..., $|\cap_{1\le i\le k} A_i|=1$. 
		We obtain 
		\begin{align*}
			N'_k(p) & = \sum_{1\le i\le k} |A_i| - \sum_{1\le i< j\le k} |A_i\cap A_j| + \sum_{1\le i< j< \ell\le k} |A_i\cap A_j\cap A_{\ell}|  - \cdots +(-1)^{k-1} |\cap_{1\le i\le k} A_i|  \\
			& = \binom{k}{1}p^{k-2} - \binom{k}{2}p^{k-3}+ \binom{k}{3}p^{k-4}- \cdots +(-1)^{k-1} \binom{k}{k} \\
			& = p^{k-1} - \frac1{p}\left((p-1)^{k}+(-1)^k (p-1)\right),
		\end{align*}
		and the proof is complete.
	\end{proof}
	
	\begin{lemma} \label{Lemma_cong_system_prime_power} Let $p^a$ be a prime power \textup{($a \ge 1 $)}. The number $N'_2(p^a)$ of solutions $(x_1,x_2)\in \{1,\ldots,p^a\}^2$ of the system of congruences  
		$x_1x_2 \equiv 0$ \textup{(mod $p^a$)}, $x_1+x_2\equiv 0$ \textup{(mod $p^a$)} is
		\begin{align*}
			N'_2(p^a)= p^{\lfloor a/2 \rfloor}.    
		\end{align*}
	\end{lemma}
	
	\begin{proof}[Proof of Lemma {\rm \ref{Lemma_cong_system_prime_power}}]
		Assume that $p^a \mid x_1x_2$ and $p^a \mid x_1+x_2$, where $1\le x_1,x_2\le p^a$. Let $p^b\mid \mid x_1$, where
		$0\le b\le a$. Then $x_1=p^by_1$ with $1\le y_1\le p^{a-b}$, $p\nmid y_1$, giving $\varphi(p^{a-b})$ choices of $y_1$
		for a fixed $b$. We deduce that $p^{a-b}\mid y_1x_2$, hence $p^{a-b}\mid x_2$ and $x_2=p^{a-b}y_2$, where $1\le y_2\le p^b$.
		
		From condition $p^a \mid x_1+x_2$ we have $p^{a-b}y_2\equiv - p^by_1$ (mod $p^a$). This is a linear congruence in $y_2$
		for every fixed $y_1$, it has solutions if and only if $\gcd(p^{a-b},p^a)=p^{a-b}\mid p^by_1$, that is $a-b\le b$, so $b\ge a/2$ (since $p\nmid y_1$). If this condition is satisfied, then there is one solution (mod $p^b$), namely $y_2 \equiv - p^{2b-a} y_1$ (mod $p^b$).
		
		The number of solutions of the system is 
		\begin{equation*}
			\sum_{b=\lceil a/2\rceil}^a \varphi(p^{a-b})= \sum_{c=0}^{\lfloor a/2 \rfloor} \varphi(p^c)= p^{\lfloor a/2 \rfloor}, 
		\end{equation*}
		ending the proof.
	\end{proof}
	
	\begin{lemma} \label{Lemma_linear}
		Let $a_1,\ldots,a_k,b\in \Z$ and $n\in \N$. The congruence $a_1x_1+\cdots +a_kx_k\equiv b$ \textup{(mod $n$)} has a solution 
		if and only if $\delta \mid b$, where $\delta =\gcd(a_1,\ldots,a_k,n)$. In this case, the number of solutions is $N_k(n)= \delta n^{k-1}$. 
	\end{lemma}
	
	\begin{proof}[Proof of Lemma {\rm \ref{Lemma_linear}}]
		See, e.g., Lehmer \cite[pp. 155--156]{Leh1913}, McCarthy \cite[Prop.\ 3.1]{McC1986}, proved by induction on $k$. 
		For the sake of completeness, here we give a direct proof using exponential sums. 
		We have by using the familiar identity 
		\begin{equation} \label{id_exp}
			\sum_{1\le a\le n} e^{2\pi ija/n} =\begin{cases} n, \text{ if $n\mid j$},\\ 0, \text{ otherwise},  
			\end{cases}    
		\end{equation}
		valid for $n,j\in \N$, 
		\begin{equation*}
			N_k(n)= \sum_{\substack{1\le x_1,\ldots,x_k\le n \\ a_1x_1+\cdots+a_kx_k\equiv b 
					\text{ (mod $n$)}}} 1=
			\frac1{n} \sum_{1\le x_1,\ldots,x_k\le n} \sum_{j=1}^n e^{2\pi ij(a_1x_1+\cdots+a_kx_k-b)/n}
		\end{equation*}
		\begin{equation*}
			=\frac1{n} \sum_{j=1}^n e^{-2\pi ijb/n} \sum_{1\le x_1\le n} e^{2\pi ij a_1x_1/n} \cdots \sum_{1\le x_k\le n} e^{2\pi ij a_kx_k/n}
		\end{equation*}
		\begin{equation*}
			=n^{k-1} \sum_{\substack{j=1\\ n\mid ja_1,\ldots, n\mid ja_k}}^n e^{-2\pi ijb/n}  = 
			n^{k-1} \sum_{\substack{j=1\\ n\mid j\gcd(a_1,\ldots,a_k)}}^n e^{-2\pi ijb/n},
		\end{equation*}
		by \eqref{id_exp} again. Let $\delta= \gcd(a_1,\ldots,a_k,n)$. Condition $n\mid j\gcd(a_1,\ldots,a_k)$ is
		equivalent to $\frac{n}{\delta} \mid j \frac{\gcd(a_1,\ldots,a_k)}{\delta}$ and to $\frac{n}{\delta} \mid j$, since $\frac{n}{\delta}$ and $\frac{\gcd(a_1,\ldots,a_k)}{\delta}$ are relatively prime. Let $j=\frac{n}{\delta}\ell\le n$, where $\ell\le \delta$. We deduce that
		\begin{equation*}
			N_k(n)= n^{k-1} \sum_{\ell=1}^{\delta} e^{-2\pi i\ell b/\delta}  = \begin{cases} \delta n^{k-1}, & \text{ if $\delta \mid b$},\\ 0, & \text{ otherwise},\end{cases}
		\end{equation*}
		completing the proof.
	\end{proof}
	
	\begin{lemma} \label{Lemma_quadratic}
		Let $k\in \N$, $\a=(a_1,\ldots,a_k)\in \Z^k$ and $p$ be a prime such that $\gcd(a_1\cdots a_k,p)=1$. If $p>2$, then the number of solutions of the congruence $a_1x_1^2+\cdots + a_kx_k^2\equiv 0$ \textup{(mod $p$)} is 
		\begin{align*}
			p^{k-1} \left(1 + \frac{p-1}{p^{k/2}} (-1)^{k(p-1)/4} \left(\frac{a_1 \cdots a_k}{p} \right) \right),
		\end{align*}
		if $k$ is even, where $\left( \frac{a}{p} \right)$ is the Legendre symbol, and $p^{k-1}$ if $k$ is odd. If $p=2$, then the number of solutions is $2^{k-1}$ for every $k\in \N$.
	\end{lemma}
	
	\begin{proof}[Proof of Lemma {\rm \ref{Lemma_quadratic}}] Formulas for the number of quadratic congruences are known in the literature. See Section \ref{Sect_Further_remarks} for some more details. However, we prefer to give here a direct proof.
		
		If $p=2$ and $a_1,\ldots,a_k$ are odd, then it directly follows that the number of solutions of the congruence $a_1x_1^2+\cdots+ a_kx_k^2\equiv 0$ (mod $2$) is $2^{k-1}$ for every $k$ (the values $x_1,\ldots, x_{k-1}\in \{1,2\}$
		can be arbitrarily selected and then for $x_k$ we have only one choice). 
		
		Now assume that $p>2$. Then the number $U_k(p)$ of solutions of the given congruence is
		\begin{equation*}
			U_k(p)= \sum_{\substack{1\le x_1,\ldots,x_k\le p \\ a_1x_1^2+\cdots+a_kx_k^2\equiv 0 \text{ (mod $p$)}}} 1=
			\frac1{p} \sum_{1\le x_1,\ldots,x_k\le p} \sum_{j=1}^p e^{2\pi ij(a_1x_1^2+\cdots+a_kx_k^2)/p}
		\end{equation*}
		\begin{equation*}
			= \frac1{p} \sum_{j=1}^p \sum_{1\le x_1\le p} e^{2\pi ija_1x_1^2/p} \cdots \sum_{1\le x_k\le p} e^{2\pi ija_kx_k^2/p} 
		\end{equation*}
		\begin{equation*}
			= \frac1{p} \left(\sum_{j=1}^{p-1} S(ja_1,p)\cdots S(ja_k,p) + p^k\right),  
		\end{equation*}
		where
		\begin{align*}
			S(j,p):= \sum_{x=1}^p  e^{2\pi ijx^2/p} = c_p \left(\frac{j}{p} \right)\sqrt{p},
		\end{align*}
		are the quadratic Gauss sums for the primes $p>2$ with $p\nmid j$, where $c_p=1$ for $p\equiv 1$ (mod $4$), 
		and $c_p=i$ for $p\equiv -1$ (mod $4$); see, e.g., Berndt et al. \cite[Ch.\ 1]{BEW1998}, Hua \cite[Sect.\ 7.5]{Hua1982}.
		This gives
		\begin{align*}
			U_k(p) =  \frac1{p} \left(\sum_{j=1}^{p-1} c_p^k \left(\frac{j}{p} \right)^k 
			\left(\frac{a_1\cdots a_k}{p} \right)(\sqrt{p})^k + p^k\right),  
		\end{align*}
		\begin{align*}
			=  \frac1{p}\left( c_p^k \left(\frac{a_1\cdots a_k}{p} \right)(\sqrt{p})^k \sum_{j=1}^{p-1} \left(\frac{j}{p} \right)^k  + p^k\right).  
		\end{align*}
		
		Here 
		\begin{align} \label{sum_Legendre}
			\sum_{j=1}^{p-1}  \left(\frac{j}{p} \right)^k = \begin{cases} \sum_{j=1}^{p-1}  1 = p-1, & \text{ if $k$ is even},\\
				\sum_{j=1}^{p-1} \left(\frac{j}{p} \right) =0, & \text{ if $k$ is odd}, \end{cases}
		\end{align}
		and $c_p^k= (-1)^{k(p-1)/4}$ for $k$ even. This gives the result.
	\end{proof}
	
	\begin{lemma} \label{Lemma_syst_cong_quadr} Let $k\in \N$ and let $p$ be a prime. Let $T_k(p)$ denote the number of solutions 
		of the system of congruences $x_1\cdots x_k\equiv 0$ \textup{(mod $p$)}, $x_1^2+\cdots + x_k^2\equiv 0$ \textup{(mod $p$)}. If $p>2$, then
		\begin{align} \label{T_k_p}
			T_k(p)= \frac1{p} \left(p^k - (p-1)^k + (-1)^{k-1}(p-1) \sum_{s=0}^{\lfloor(k-1)/2\rfloor} (-1)^{s(p-1)/2} \binom{k}{2s} p^s \right).
		\end{align}
		If $p=2$, then $T_k(2)=2^{k-1}$ for $k$ odd and $T_k(2)=2^{k-1}-1$ for $k$ even.
	\end{lemma}

	\begin{proof}[Proof of Lemma {\rm \ref{Lemma_syst_cong_quadr}}]
		Here $T_k(p)$ is the difference between the number of all solutions of the congruence
		$x_1^2+\cdots + x_k^2\equiv 0$ (mod $p$) and the number of its solutions, where $x_1,\ldots,x_k\in \{1,\ldots,p-1\}$.
		If $p>2$, then we obtain $T_k(p)=V_k(p)-V'_k(p)$, where
		\begin{align*}
			V_k(p)= \sum_{\substack{1\le x_1,\ldots,x_k\le p \\ x_1^2+\cdots+x_k^2\equiv 0 \text{ (mod $p$)}}} 1= 
			\frac1{p} \left(\sum_{j=1}^{p-1} (S(j,p))^k + p^k\right),  
		\end{align*}
		as already included in the proof of Lemma \ref{Lemma_quadratic} (taking $a_1=\cdots =a_k=1$), and similarly,
		\begin{align*}
			V'_k(p)= \sum_{\substack{1\le x_1,\ldots,x_k\le p-1 \\ x_1^2+\cdots+x_k^2\equiv 0 \text{ (mod $p$)}}} 1 
			= \frac1{p} \left(\sum_{j=1}^{p-1} (S(j,p)-1)^k + (p-1)^k\right). 
		\end{align*}
		
		Thus
		\begin{align*}
			V_k(p)- V'_k(p) = \frac1{p} U_k(p) + \frac1{p}\left(p^k - (p-1)^k \right),
		\end{align*} 
		where
		\begin{align*}
			U_k(p) & = \sum_{j=1}^{p-1} \left(S(j,p))^k - (S(j,p)-1)^k\right) \\ 
			& = (-1)^{k-1} \sum_{j=1}^{p-1} \sum_{t=0}^{k-1} (-1)^t \binom{k}{t} (S(j,p))^t \\
			& = (-1)^{k-1} \sum_{j=1}^{p-1} \sum_{t=0}^{k-1} (-1)^t \binom{k}{t} \left(c_p \left(\frac{j}{p} \right) 
			\sqrt{p}\right)^t \\
			& =(-1)^{k-1} \sum_{t=0}^{k-1} (-c_p)^t \binom{k}{t} \sqrt{p}^t \sum_{j=1}^{p-1}  \left(\frac{j}{p} \right)^t.
		\end{align*}
		
		Using identity \eqref{sum_Legendre} we deduce that
		\begin{align*}
			U_k(p)= (-1)^{k-1}(p-1) \sum_{s=0}^{\lfloor(k-1)/2\rfloor} (-1)^{s(p-1)/2} \binom{k}{2s} p^s,
		\end{align*}
		giving \eqref{T_k_p}.
		
		Assume that $p=2$. Then $(1,\ldots,1)$ is a solution of $x_1^2+\cdots+x_k^2\equiv 0$ (mod $2$) if and only if $k$ is even. 
		If $k$ is odd, then for an arbitrary choice of $x_1,\ldots,x_{k-1}\in \{1,2\}$, the value of $x_k$ is determined to satisfy this
		congruence, giving $2^{k-1}$ solutions of the system. If $k$ is even, then the same is true, but $(1,\ldots,1)$ is not a solution of the system, so we have $2^{k-1}-1$ solutions. 
	\end{proof}
	
	\section{Proofs of the theorems} \label{Section_Proofs}
	
	\begin{proof}[Proof of Theorem {\rm \ref{Th_main}}]
		We use the following basic property of the M\"obius function:
		\begin{align*}
			\sum_{d\mid n} \mu(d)= \begin{cases}  1, & \text{ if $n=1$}, \\
				0, & \text{ if $n>1$}.
			\end{cases}
		\end{align*}
		
		We have from the definition of $\Phi_F(n)$, 
		\begin{align*}
			\Phi_F(n) & = \sum_{1\le x_1,\ldots,x_k\le n} \sum_{d\mid \gcd(f_1(x_1,\ldots, x_k), \ldots, f_m(x_1,\ldots,x_k),n)}  \mu(d) \\
			& = \sum_{1\le x_1,\ldots,x_k\le n} \sum_{\substack{d\mid n\\ d\mid f_i(x_1,\ldots, x_k), 1\le i\le m}} \mu(d) \\
			& = \sum_{d\mid n} \mu(d) \sum_{\substack{1\le x_1,\ldots,x_k\le n \\  f_i(x_1,\ldots, x_k)\equiv 0 \text{ (mod $d$)}, 1\le i\le m}} 1. 
		\end{align*}
		
		Here the inner sum is $N_F(d) \left(\frac{n}{d}\right)^k$. Indeed, let $(x_1^0,\ldots,x_k^0)\in \{1,\ldots, d\}^k$ be a solution 
		of the system of congruences $f_i(x_1,\ldots, x_k)\equiv 0 \text{ (mod $d$)}$ ($1\le i\le m$). Then for every $(j_1,\ldots,j_k)\in \{0,\ldots,n/d-1\}^k$ the $k$-tuples $(x_1^0+j_1d,\ldots,x_k^0+j_kd)\in \{1,\ldots, n\}^k$ are solutions of the system of congruences $f_i(x_1,\ldots, x_k)\equiv 0 \text{ (mod $n$)}$ ($1\le i\le m$). Conversely, every solution of the system of congruences (mod $n$) can be obtained from the system (mod $d$), where $d\mid n$.
		
		This gives identity \eqref{eq_1}. Since the function $N_F(n)$ is multiplicative, we deduce that $\Phi_F(n)$ is also 
		multiplicative, being the convolution of multiplicative functions, and this  property immediately leads to \eqref{eq_2}.
	\end{proof}
	
	\begin{proof}[Proof of Theorem {\rm \ref{Th_main_h}}] 
		We use that $h(n)=\sum_{d\mid n} (\mu*h)(d)$ for every $n\in \N$. Similarly to the proof of Theorem \ref{Th_main}, 
		\begin{align*}
			P_{F,h}(n) &  = \sum_{1\le x_1,\ldots,x_k\le n} \ \sum_{d\mid \gcd(f_1(x_1,\ldots, x_k), \ldots, f_m(x_1,\ldots,x_k),n)}  (\mu*h)(d) \\
			& = \sum_{1\le x_1,\ldots,x_k\le n} \sum_{\substack{d\mid n\\ d\mid f_i(x_1,\ldots, x_k), 1\le i\le m}} (\mu*h)(d) \\
			& = \sum_{d\mid n} (\mu*h)(d) \sum_{\substack{1\le x_1,\ldots,x_k\le n \\  f_i(x_1,\ldots, x_k)\equiv 0 \text{ (mod $d$)}, 1\le i\le m}} 1 \\
			& = \sum_{d\mid n} (\mu*h)(d) N_F(d) \left(\frac{n}{d}\right)^k, 
		\end{align*}
		leading to identity \eqref{form_h}.
		
		If the function $h$ is multiplicative, then $P_{F,h}(n)$ is also multiplicative, being the convolution of multiplicative functions.
	\end{proof}
	
	\begin{proof}[Proof of Theorem {\rm \ref{Th_asympt}}] 
		According to \eqref{eq_1} and the first estimate of Lemma \ref{Lemma_sum} we have 
		\begin{equation*}
			\sum_{n \leq x} \Phi_F(n) = \sum_{de=n\leq x} \mu(d) N_F(d) e^k = 
			\sum_{d \leq x} \mu(d) N_F(d) \sum_{e \leq x/d} e^k
		\end{equation*}
		\begin{equation*}
			= \sum_{d \leq x} \mu(d) N_F(d) \left( \frac{1}{k+1} \left(\frac{x}{d}\right)^{k+1} + O\left(\left(\frac{x}{d}\right)^k\right) \right)
		\end{equation*}
		\begin{equation*}
			= \frac{x^{k+1}}{k+1} \sum_{d \leq x} \frac{\mu(d) N_F(d)}{d^{k+1}} + O\left(x^k \sum_{d \leq x} \frac{N_F(d)}{d^k}\right)
		\end{equation*}
		\begin{equation*}
			= \frac{x^{k+1}}{k+1} \sum_{d=1}^\infty \frac{\mu(d) N_F(d)}{d^{k+1}} + O\left(x^{k+1} \sum_{d > x} \frac{N_F(d)}{d^{k+1}}\right) + O\left(x^k \sum_{d \leq x} \frac{N_F(d)}{d^k}\right),
		\end{equation*}
		where the series is absolutely convergent, since by Lemma \ref{Lemma_Hua_kov} we have 
		\begin{equation*}
			\frac{|\mu(d)| N_F(d)}{d^{k+1}}\ll \frac1{d^{1+1/t-\varepsilon}} ,
		\end{equation*}
		and choose $0<\varepsilon<1/t$. Here
		\begin{equation*}
			\sum_{d=1}^\infty \frac{\mu(d) N_F(d)}{d^{k+1}}
			= \prod_p \sum_{\nu=0}^\infty \frac{\mu(p^\nu) N_F(p^\nu)}{(p^\nu)^{k+1}} = \prod_p \left(1 - \frac{N_F(p)}{p^{k+1}}\right),
		\end{equation*}
		writing as an Euler product, applying the multiplicativity of the function $N_F(n)$. Also, by Lemma \ref{Lemma_sum},
		\begin{equation*}
			x^k \sum_{d \leq x} \frac{N_F(d)}{d^k}\ll x^k \sum_{d \leq x} \frac{d^{k-1/t+\varepsilon}}{d^k}=
			x^k \sum_{d\le x} d^{-1/t+\varepsilon} \ll x^{k+1-1/t+\varepsilon},
		\end{equation*}
		\begin{equation*}
			x^{k+1} \sum_{d > x} \frac{N_F(d)}{d^{k+1}} \ll x^{k+1} \sum_{d >x} \frac{d^{k-1/t+\varepsilon}}{d^{k+1}}=
			x^{k+1} \sum_{d> x} d^{-1-1/t+\varepsilon} \ll x^{k+1-1/t+\varepsilon},
		\end{equation*}
		giving the error term of \eqref{asympt_form}.
	\end{proof} 
	
	\begin{proof}[Proof of Theorem {\rm \ref{Th_varphi_lin_gen}}] This is an immediate consequence of Lemma \ref{Lemma_syst_lin_cong} and identity \eqref{eq_2}.
	\end{proof}
	
	\begin{proof}[Proof of Theorem {\rm \ref{Th_varphi_k}}]
		Let $p$ be a prime. By Lemma \ref{Lemma_cong_prod} the number of solutions $N_k(p)$ of the
		congruence $x_1\cdots x_k(x_1+\cdots +x_k)\equiv 0$ (mod $p$) is 
		\begin{equation*}
			N_k(p) = p^k- \frac{p-1}{p}\left((p-1)^k+(-1)^{k-1}\right),
		\end{equation*}
		giving
		\begin{equation*}
			1-\frac{N_k(p)}{p^k} =\left(1-\frac1{p}\right) \left(\left(1-\frac{1}{p}\right)^k - \frac{(-1)^{k}}{p^k} \right),
		\end{equation*}
		
		Now by \eqref{eq_2} we deduce formula \eqref{varphi_k_form}. 
	\end{proof}
	
	\begin{proof}[Proof of Theorem {\rm \ref{Th_varphi_k_apost}}] By Lemma \ref{Lemma_cong_system}
		the number of solutions of the corresponding system of congruences is
		\begin{equation*}
			N'_k(p)= p^{k-1} - \frac1{p}\left((p-1)^k+(-1)^k(p-1)\right).    
		\end{equation*}
		
		This gives 
		\begin{equation*}
			1-\frac{N'_k(p)}{p^k} =\left(1-\frac1{p}\right) \left(1+ \frac{1}{p}\left(1-\frac1{p}\right)^{k-1} + \frac{(-1)^{k}}{p^k} \right),
		\end{equation*}
		and by \eqref{eq_2} we obtain \eqref{Varphi_k_apost}.
	\end{proof}
	
	\begin{proof}[Proof of Theorem {\rm \ref{Th_Phi_k_linear}}]
		Apply Lemma \ref{Lemma_linear} for $n=p$, a prime. Let $N_k(\a,b,p)$ be the number of solutions of the congruence
		$a_1x_1+\cdots +a_kx_k+b\equiv 0$ (mod $p$). We deduce that 
		
		Case 1. Let $p\mid \gcd(a_1,\ldots,a_k)$. Then $\gcd(a_1,\ldots,a_k,p)=p$ and $N_k(\a,b,p)=p^k$ if $p\mid b$,
		and $N_k(\a,b,p)=0$ if $p\nmid b$.
		
		Case 2. If $p\nmid \gcd(a_1,\ldots,a_k)$, then $\gcd(a_1,\ldots,a_k,p)=1$ and $N_k(\a,b,p)=p^{k-1}$ for all values of $b$.
		
		Now, \eqref{id_linear} follows by \eqref{eq_2}.
	\end{proof}
	
	\begin{proof}[Proof of Theorem {\rm \ref{Th_Phi_k_squares}}]
		Apply the identities deduced in Lemma \ref{Lemma_quadratic} for the number $N_k(\a,0,p)$ of solutions of the congruence
		$a_1x_1^2+\cdots +a_kx_k^2\equiv 0$ (mod $p$). If $p>2$ and $k$ is even, then
		\begin{align*}
			1- \frac{N_k(\a,0,p)}{p^k}  & = 1 - \frac1{p} \left(1+ \frac{p-1}{p^{k/2}} (-1)^{k(p-1)/4} \left( \frac{a_1\cdots a_k}{p} \right) \right) \\
			& = \left(1-\frac1{p}\right) \left(1- \frac1{p^{k/2}} (-1)^{k(p-1)/4} \left( \frac{a_1\cdots a_k}{p} \right)\right). 
		\end{align*}
		
		If $p>2$ and $k$ is odd, then
		\begin{equation*}
			1- \frac{N_k(\a,0,p)}{p^k}   = 1 - \frac1{p}.
		\end{equation*} 
		
		Since for $p=2$ we have $N_k(\a,0,2)=2^{k-1}$ for all values of $k$,
		using \eqref{eq_2} we obtain identities \eqref{id_square_odd} and \eqref{id_square_even}.
	\end{proof}
	
	\begin{proof}[Proof of Theorem {\rm \ref{Th_Phi_k_last}}]
		Use Lemma \ref{Lemma_syst_cong_quadr} and identity \eqref{eq_2}. 
	\end{proof}
	
	\begin{proof}[Proof of Theorem {\rm \ref{Th_main_unitary}}]
		We have, by using that $d\mid \mid (a,b)_*$ holds if and only if $d\mid a$ and $d\mid \mid b$, 
		\begin{align*}
			\Phi_F^*(n) & = \sum_{1\le x_1,\ldots,x_k\le n} \sum_{d\mid (\gcd(f_1(x_1,\ldots, x_k), \ldots, f_m(x_1,\ldots,x_k)),n)_*}  
			\mu^*(d) \\
			& = \sum_{1\le x_1,\ldots,x_k\le n} \sum_{\substack{d\mid \mid n\\ d \mid f_i(x_1,\ldots, x_k), 1\le i\le m}} \mu^*(d) \\
			& = \sum_{d\mid \mid  n} \mu^*(d) \sum_{\substack{1\le x_1,\ldots,x_k\le n \\  f_i(x_1,\ldots, x_k)\equiv 0 \text{ (mod $d$)}, 1\le i\le m}} 1 \\
			& = \sum_{d\mid \mid n} \mu^*(d) N_F(d) \left(\frac{n}{d}\right)^k, 
		\end{align*}
		similarly to the proof of Theorem \ref{Th_main}, giving \eqref{eq_1_unit}. Since the function $N_F(n)$ is multiplicative, we deduce that $\Phi_F^*(n)$ is also multiplicative, being the unitary convolution of multiplicative functions, and  
		this leads to \eqref{eq_2_unit}.
	\end{proof} 
	
	\begin{proof}[Proof of Theorem {\rm \ref{Th_varphi_unit_lin_gen}}] 
		Use Lemma \ref{Lemma_syst_lin_cong} and \eqref{eq_2_unit}.
	\end{proof} 
	
	\begin{proof}[Proof of Theorem {\rm \ref{Th_unitary_2}}] 
		Use Lemma \ref{Lemma_cong_system_prime_power} and \eqref{eq_2_unit}.
	\end{proof} 
	
	\begin{proof}[Proof of Theorem {\rm \ref{Th_Phi_k_linear_unit}}]
		Similarly to the proof of the above identities, by using Lemma \ref{Lemma_linear}.
	\end{proof}
	
	\section{Some more remarks and further research} \label{Sect_Further_remarks}
	
	Formulas concerning the number $N_k(\a,b,n)$ of solutions of the quadratic congruence $a_1x_1^2+\cdots + a_kx_k^2\equiv b$ 
	(mod $n$), where $a_1,\ldots,a_k,b\in \Z$ and $n\in \N$, are known in the literature. See, e.g., the survey by the second author \cite{Tot2014}, including direct proofs. Also, see Grau and Oller-Marc\'{e}n \cite{GM2019}. 
	
	To deduce a formula for the unitary function $\Phi^{(2)*}_{k,\a,0}(n)$ (case $b=0$) one can use the following identities. If $k$ is odd, $n$ is odd and $\gcd(a_1\cdots a_k,n)=1$, then by \cite[Prop.\ 19]{Tot2014},
	\begin{align*}
		N_k(\a,0,n)= n^{k-1} \sum_{d^2\mid n} \frac{\varphi(d)}{d^{k-1}}.
	\end{align*} 
	
	If $k$ is even, $n$ is odd and $\gcd(a_1\cdots a_k,n)=1$, then according to \cite[Prop.\ 4]{Tot2014},  
	\begin{align*}
		N_k(\a,0,n)= n^{k-1} \sum_{d\mid n} \frac{\varphi(d)}{d^{k/2}} \left( \frac{(-1)^{k/2} a_1\cdots a_k}{d} \right),
	\end{align*}
	where $\left( \frac{\DOT}{d} \right)$ denotes the Jacobi symbol.
	
	By taking $n=p^\nu$, where $p>2$ is a prime and $\nu \in \N$, explicit formulas for $N_k(\a,0,p^{\nu})$
	can be given. The case $p=2$ can also be treated. Now, \eqref{eq_2_unit} leads to explicit formulas computing 
	$\Phi^{(2)*}_{k,\a,0}(n)$.
	
	The case $s\ge 3$ of the functions $\Phi^{(s)}_{k,\a,b}(n)$ and $\Phi^{(s)*}_{k,\a,b}(n)$ can also be treated.
	Williams \cite{Wil1966} deduced an explicit but complicated formula for the number of solutions of the congruence $a_1x_1^3+\cdots +a_kx_k^3+b\equiv 0$ (mod $p$), with $a_1,\ldots, a_k\in \Z$, $p\nmid a_1\cdots a_k$. According to \eqref{eq_2} it leads to a formula to compute $\Phi^{(3)}_{k,\a,b}$(n). For the general case $s\ge 3$ and for the number of solutions of some more general congruences see, e.g, Berndt et al. \cite{BEW1998}, Hull \cite{Hul1932}, 
	Ireland and Rosen \cite{IR1990}, Lemmermeyer \cite{Lem2000}, Li and Ouyang \cite{LO2018}. 
	
	See Nilsson and Nyqvist \cite{NN2021} for an approach to generalize Lemma \ref{Lemma_syst_lin_cong} in order to 
	deduce the number of solutions of the linear congruence system 
	$a_{i1}x_1+\cdots +a_{ik}x_k+b_i \equiv 0$ (mod $n$) with $a_{ij},b_i\in \Z$ ($1\le i,j\le k$). This concerns the functions
	$\phi_{A,B}(n)$ and $\phi^*_{A,B}(n)$.
	
	We pose as open problems to improve the error term in the asymptotic formula of Theorem \ref{Th_asympt} concerning the function
	$\Phi_F(n)$ and to deduce an asymptotic formula for the unitary function $\Phi_F^*(n)$.

\end{document}